\documentclass[12pt]{amsart}

\usepackage{amsmath}
\usepackage{amsfonts}
\usepackage{amssymb,enumerate}
\usepackage{amsthm}
\usepackage{fancyhdr}
\usepackage{tikz}
\usepackage{verbatim}
\usetikzlibrary{arrows,matrix}
\usepackage{hyperref}

\usepackage{caption}
\usepackage[noend]{algpseudocode}
\usepackage{epstopdf}

\newtheorem{lemma}{Lemma}
\newtheorem{corollary}[lemma]{Corollary}
\newtheorem{proposition}[lemma]{Proposition}
\newtheorem{theorem}[lemma]{Theorem}

\title[]{Minimal relations and the Diophantine Frobenius Problem in embedding dimension three}

\author{Alessio Moscariello}

\subjclass[2010]{11D75, 11D07, 20M14}

\keywords{Frobenius problem, numerical semigroups, minimal relations}

\address[Alessio Moscariello]{Dipartimento di Matematica e Informatica, \ Universit\`a di Catania, \  Viale Andrea Doria 6, 
95125 Catania,Italy}

\email{alessio.moscariello@studium.unict.it}

\begin{document}
\maketitle
\begin{abstract}
This paper provides a formula for the minimal relations and the Frobenius number of a numerical semigroup minimally generated by three pairwise coprime positive integers.
\end{abstract}

\section*{Introduction}

Let $n_1,\ldots,n_\nu$ be a set of positive integers such that $\gcd(n_1,\ldots,n_\nu)=1$. The \emph{Diophantine Frobenius Problem}  (cf. \cite{Br}) asks for an explicit formula that determines the largest integer which cannot be expressed as a linear combination with nonnegative integer coefficients of $n_1,\ldots,n_\nu$, usually called \emph{Frobenius number} and denoted by $F(n_1,\ldots,n_\nu)$. This problem dates back to the end of the nineteenth century: a classical result of Sylvester (cf. \cite{S}) solved the problem for $\nu=2$, giving the well-known formula $F(n_1,n_2)=n_1n_2-n_1-n_2.$
However, as $\nu$ grows the problem becomes incredibly complicated: in fact today the problem is still open for $\nu \ge 3$ (\cite{RA} is an excellent monograph in that regard). The case $\nu=3$ has been the main subject of a wide number of papers, and many formulas for particular triples have been found, although the main case still remains unsolved. A noteworthy result of Curtis (cf. \cite{C}) states that it is impossible to find a polynomial formula for $F(n_1,n_2,n_3)$.

Several results (cf. \cite{J}, \cite{RA}, \cite{RG}) established a simple relation between the Frobenius number $F(n_1,n_2,n_3)$ and the  \emph{minimal relations} between $n_1,n_2,n_3$, i.e. the smallest positive integers $c_i$ such that $c_in_i$ is representable as $c_in_i=\lambda_jn_j+\lambda_kn_k$, where $\lambda_j, \lambda_k \in \mathbb{N}$; therefore, the attention is generally focused on finding formulas or algorithms for the minimal relations $c_1,c_2,c_3$. Several algorithms are present in the literature (cf. \cite{M}, \cite{RV}), which compute these relations with the same complexity as Euclid's greatest common divisor algorithm.\\
The main result of this short paper is a generic formula for the minimal relations, which determines $c_i$ as the minimum element of a certain finite set. Then, a simple application of this result provides a (terrible) non-polynomial formula for $F(n_1,n_2,n_3)$ in function of the minimal generators $n_1,n_2,n_3$ which involves the operators $\min,\max, \lfloor \cdot \rfloor, \lceil \cdot \rceil$ and the remainder operator $[\cdot]_\cdot$, which partly explains its irregular behaviour. 

In Section 1 we recall several preliminary results, in order to define our context and give the necessary background, while Section 2 is devoted to prove the main result (Theorem \ref{main}). At the end of this Section, the formula for $F(n_1,n_2,n_3)$ is stated (Theorem \ref{iterfrob}). Finally, we conclude this work with some considerations on the results obtained.
\section{Preliminaries}
Let $\mathbb{N}$ denote the set of nonnegative integers. A \emph{numerical semigroup} is a submonoid $\mathcal{S}$ of $(\mathbb{N},+)$ such that $\mathbb{N} \setminus \mathcal{S}$ is finite. Each numerical semigroup admits a finite set of minimal generators $\mathcal{G}$ (we use the standard notation $\mathcal{S}=\langle \mathcal{G} \rangle$), and it is easy to deduce that the condition on $\mathbb{N} \setminus \mathcal{S}$ is equivalent to $\gcd(\mathcal{G})=1$. The \emph{Frobenius number} of a numerical semigroup $\mathcal{S}$ is the largest element of $\mathbb{N}\setminus \mathcal{S}$, and is commonly denoted by $F(\mathcal{S})$. 

The first Theorem we need is based on the classical result of Sylvester (cf. \cite{S},\cite{RG}), and studies numerical semigroups minimally generated by two positive integers.

\begin{theorem}\label{class}
Let $n_1,n_2$ be two coprime positive integers greater than $2$, and let $S=\langle n_1,n_2 \rangle$. Then $$F (\mathcal{S}) = n_1n_2-n_1-n_2.$$
Moreover, if $m \in \mathbb{N} \setminus \mathcal{S}$, then $F(\mathcal{S})-m \in \mathcal{S}$.
\end{theorem}

This result plays, surprisingly, a crucial role in our investigation; in particular, the implication $m \not \in \mathbb{N} \setminus \mathcal{S} \Longrightarrow F(\mathcal{S})-m \in \mathcal{S}$ provides the basic idea of our approach.

Now, let $n_1,n_2,n_3$ be positive integers such that $\gcd(n_1,n_2,n_3)=1$, and consider the numerical semigroup $\mathcal{S}=\langle n_1,n_2,n_3 \rangle$. From now on, we assume that ${n_1,n_2,n_3}$ is the minimal system of generators for $\mathcal{S}$ (i.e. each $n_i$ is not representable as a linear combination of the other two). The following Proposition reduces our problem to the case of pairwise coprime integers $n_1,n_2,n_3$.

\begin{proposition}[\protect{\cite[Lemma 2.16]{RG}}]\label{coprime}
Let $\mathcal{S}$ be a numerical semigroup minimally generated by three positive integers $n_1,n_2,n_3$ such that $\gcd(n_1,n_2,n_3)=1$. Let $d=\gcd(n_1,n_2)$, and consider the numerical semigroup $\mathcal{T}=\langle \frac{n_1}{d},\frac{n_2}{d},n_3 \rangle$. Then $$F(\mathcal{S})=dF(\mathcal{T})+(d-1)n_3.$$
\end{proposition}

Let $\{i,j,k\} = \{1,2,3\}$. Assuming that $n_1,n_2,n_3$ are pairwise coprime, the sets $\langle n_j,n_k \rangle$ are numerical semigroups.
Define the sets $$\mathcal{S}_i = \{M \in \mathbb{N} \setminus \{0\} \ | \ Mn_i \in \langle n_j, n_k \rangle \},$$ 
and define $c_i = \min \mathcal{S}_i \setminus \{0\}$; the integers $c_i$ are the \emph{minimal relations} discussed in the introduction.

This definition is well-posed, since clearly the product $n_jn_i$ belongs to $\langle n_j,n_k \rangle$; actually, this remark implies $c_i \le \min \{n_j,n_k\}$. Notice that, in this context, the assumption that $n_1,n_2,n_3$ form a minimal system of generators for $\mathcal{S}$ is equivalent to $c_i > 1$, for every $i=1,2,3$. Actually, since the sets $\langle n_j,n_k \rangle$ are numerical semigroups, clearly $Mn_i \in \langle n_j,n_k \rangle$ if $M$ is large enough, i.e. $M \in \mathcal{S}_i$: thus $\mathbb{N} \setminus \mathcal{S}_i$ is finite, and it is easy to check that $\mathcal{S}_i$ is a numerical semigroup.

X It is possible to characterize the sets $\mathcal{S}_i$ as sets of solutions of certain modular inequalities. In the following, we denote by $[m]_n$  the remainder of the Euclidean division of $m$ by $n$.
\begin{proposition}[\protect{\cite[Corollary 2.2]{Mos}}]\label{eq}
	Let $n_1,n_2,n_3$ be pairwise coprime positive integers.
	 
	Then, with the notation fixed above, we have
	$$\mathcal{S}_i=\{ x \in \mathbb{N} \ | \ [xn_in_k^{-1}]_{n_j}n_k \le xn_i \}$$
	for every $\{i,j,k\}=\{1,2,3\}$.
\end{proposition}
Several results (cf. \cite{J},\cite{RA},\cite{RG2}) provide simple formulas for $F(\langle n_1,n_2,n_3 \rangle)$ as a function of $n_1,n_2,n_3$, the integers $c_i$ and the coefficients of the linear combination that determines $c_in_i$. The one we will use is a slight variation of a result by Johnson.

\begin{theorem}[\protect{\cite[Theorem 3.4]{Mos}}]\label{altfrob}
	Let $n_1,n_2,n_3 \in \mathbb{Z}^+$ be pairwise coprime integers such that $n_1 > n_2 > n_3 > 1$, and let $\mathcal{S}=\langle n_1,n_2,n_3 \rangle$. Then $$F(\mathcal{S})=c_1n_1 + \max\{ [c_2n_2n_3^{-1}]_{n_1}n_3, [c_3n_3n_2^{-1}]_{n_1}n_2 \}-n_1-n_2-n_3.$$
\end{theorem}


\section{Main result}
Let $\mathcal{S}$ be a numerical semigroup minimally generated by three pairwise coprime positive integers $n_1,n_2,n_3$, and let $\{i,j,k\}=\{1,2,3\}$. Our approach will use the classical Theorem \ref{class} to obtain $n_i$ from the Frobenius number $F_i=F(\langle n_j,n_k \rangle)=n_jn_k-n_j-n_k$. Since $n_i \not \in \langle n_j,n_k \rangle$, by Theorem \ref{class} $n_jn_k-n_j-n_k-n_i \in \langle n_j,n_k \rangle$; therefore there exist $A,B \in \mathbb{N}$ such that $n_jn_k-n_j-n_k-n_i=An_j+Bn_k$. From this equation we can deduce that there exist two positive integers $\lambda_{ij},\lambda_{ik}$ such that $n_i=n_jn_k-\lambda_{ij}n_j-\lambda_{ik}n_k$. 
By the hypothesis of pairwise coprimality, we can obtain $\lambda_{ij}, \lambda_{ik}$ from $n_1,n_2,n_3$, using the remainder operator. In fact, we have $-\lambda_{ij}n_j \equiv n_i \pmod{n_k}$, and, since clearly $\lambda_{ij} < n_k$, we get $\lambda_{ij}=[-n_in_j^{-1}]_{n_k}$, and similarly $\lambda_{ik} = [-n_in_k^{-1}]_{n_j}$. In order to simplify the notation, we state our results in function of $\lambda_{ij},\lambda_{ik}$, and will compute them only in the statement of the main results.

The basic idea of our approach is to \emph{travel} from $F_i$ to $n_i$ by subtracting $(\lambda_{ij}-1)n_j$, and then $(\lambda_{ik}-1)n_k$, while studying the behaviour of the sets $\mathcal{S}_i$ during each step. 
We start by subtracting $(\lambda_{ij}-1)n_j$; denote by $N_i$ the integer obtained after this first step, that is, $N_i=n_jn_k-\lambda_{ij}n_j-n_k$. Also, define $\mathcal{T}=\langle n_j,n_k,N_i \rangle$ and denote by $\mathcal{T}_i, \mathcal{T}_j, \mathcal{T}_k$ the analogous of the sets $\mathcal{S}_i, \mathcal{S}_j, \mathcal{S}_k$. Notice that $\gcd(N_i,n_j)=\gcd(n_jn_k-\lambda_{ij}n_j-n_k,n_j)=\gcd(n_k,n_j)=1$ and $\gcd(N_i,n_k)=\gcd(n_jn_k-\lambda_{ij}n_j-n_k,n_j)=\gcd(n_jn_k-\lambda_{ij}n_j-\lambda_{ik}n_k,n_j)=\gcd(n_i,n_k)=1$: the integers $n_j,n_k,N_i$ are thus pairwise coprime, and we fall in the hypothesis of the following Proposition.
\begin{proposition}\label{state}
	Let $N_i,n_j,n_k$ be pairwise coprime positive integers such that $N_i=n_jn_k-\lambda_{ij}n_j-n_k$, with $\lambda_{ij} \in \mathbb{Z}^+$, and let $\mathcal{T}=\langle n_j,n_k,N_i \rangle$.
	Then
	\begin{enumerate}
		\item $c_j=n_k-\lambda_{ij}.$
		\item $c_k=n_j - \left \lfloor \frac{n_k}{n_k-\lambda_{ij}} \right \rfloor $. Moreover, if we write$x=\alpha n_j + \beta$, with $0 < \beta \le n_j$, we have that $x \in \mathcal{T}_k$ if and only if
		$$ \beta \ge \frac{n_i-\alpha n_k}{n_k-\lambda_{ij}}=n_j-\frac{(\alpha + 1)n_k}{n_k-\lambda_{ij}}.$$
	\end{enumerate}
\end{proposition}
\begin{proof}
	\begin{enumerate}	
		\item Clearly $$(n_k-\lambda_{ij})n_j=n_kn_j-\lambda_{ij}n_j=n_i+n_k,$$ thus by the hypothesis of pairwise coprimality it follows that $c_j=n_k-\lambda_{ij}$.
		\item By Proposition \ref{eq} we have $$\mathcal{T}_k = \{x \in \mathbb{N} \ | \ xn_k \ge [xn_kn_i^{-1}]_{n_j}n_i\}.$$ However, $n_i \equiv -n_k \pmod{n_j}$, implying  $[xn_kx_i^{-1}]_{n_j}=[-x]_{n_j}$: therefore, we obtain that $x \in \mathcal{T}_k$ if and only if $xn_k \ge [-x]_{n_j}n_i$. Now, write $x=\alpha n_j + \beta$, with $0 < \beta \le n_j$. It is easy to see that $[-x]_{n_j}=n_j-\beta$, hence $\alpha n_j + \beta \in \mathcal{T}_k$ if and only if $$(\alpha n_j + \beta)n_k \ge (n_j-\beta)n_i \Longrightarrow (n_k+n_i)\beta \ge (n_i-\alpha n_k)n_j.$$ Since $n_i+n_k=n_j(n_k-\lambda_{ij})$ we obtain $$ \beta \ge \frac{n_i-\alpha n_k}{n_k-\lambda_{ij}}=n_j-\frac{(\alpha + 1)n_k}{n_k-\lambda_{ij}},$$
		which is the second part of our thesis. Finally, taking $\alpha = 0$ we deduce $$c_k= n_j - \left \lfloor \frac{n_k}{n_k-\lambda_{ij}} \right \rfloor.$$
	\end{enumerate}
\end{proof}
Notice that, when we subtract multiples of $n_j$, the behaviour of the associated integer $c_j$ is very simple. For this reason, since we have yet to subtract multiples of $n_k$, we focused on the set $\mathcal{T}_k$, in the hope that we can control this set during the second step. 

Luckily, this turns out to be the case. Going back to the general case, the following Proposition links the sets $\mathcal{S}_k$ and $\mathcal{T}_k$.
\begin{proposition}\label{pass}
	With the notation fixed above, we have
	$\mathcal{S}_k = \{x-(\lambda_{ik}-1)[-x]_{n_j} | x \in \mathcal{T}_k\}$.
\end{proposition}
\begin{proof}
Take $x \in \mathcal{S}_k$.  Then $xn_k$ can be expressed as a linear combination of $n_j$ and $n_i$, say $xn_k=An_j+Bn_i$, with $A,B \in \mathbb{N}$; actually, we can further assume without loss of generality that $0 \le B < n_j$. Since $N_i - (\lambda_{ik}-1)n_k = n_i$, we get 
\begin{equation}\label{step}
xn_k = A n_j + B n_i \Longleftrightarrow [x+B(\lambda_{ik}-1)]n_k=A n_j + B N_i,
\end{equation}
hence $x+B(\lambda_{ik}-1) \in \mathcal{T}_k$.
Then for every $x \in \mathcal{S}_k$ there exists $x' \in \mathcal{T}_k$ such that $x'n_k=A'n_j + B'N_i$, with $0 \le B' < n_j$ and $x=x'-B'(\lambda_{ik}-1)$. Conversely, it is obvious that if $x' \in \mathcal{T}_k$ is such that $x'n_k=A'n_j + B'N_i$, with $0 \le B' < n_j$, then $x'-B'(\lambda_{ik}-1) \in \mathcal{S}_k$ (notice that $A',B' \in \mathbb{N}$ implies $x'-B'(\lambda_{ik}-1) \in \mathbb{N}$). Therefore, we have obtained the following relation:
$$\mathcal{S}_k=\{x-B(\lambda_{ik}-1) \ | x \in \mathcal{T}_k, xn_k=An_j+BN_i, A \in \mathbb{N},0 \le B < n_j\}.$$

In order to conclude the proof, we deduce the coefficient $B$ of the previous relation in function of $x$: in fact, recalling that $N_i \equiv -n_k \pmod{n_j}$, from $xn_k \equiv B N_i \pmod{n_j}$ we obtain $-x \equiv B \pmod{n_j}$, thus $0 \le B < n_j$ forces $B=[-x]_{n_j}$, which is our thesis.
\end{proof}

This result shows that the final configuration of the set $\mathcal{S}_k$ depends only on the composition of the set $\mathcal{T}_k$ and the number $\lambda_{ik}-1$ of multiples of $n_k$ that divide the two generators $N_i$ and $n_i$. Since the set $\mathcal{T}_k$ is fully determined by Proposition \ref{state}, combining Propositions \ref{state} and \ref{pass} we prove our main result, that is, $c_k$ is the minimum element of a well-defined finite set. We state this result in its full generality.

\begin{theorem}\label{main}
	Let $n_1,n_2,n_3$ be three pairwise coprime positive integers which are the minimal generators of the numerical semigroup $\mathcal{S}=\langle n_1,n_2,n_3 \rangle$. Then, for every $\{i,j,k\}=\{1,2,3\}$, we have $$c_k=\min_{\alpha=1,\ldots,I_k} \left\{\alpha n_j - [-n_in_k^{-1}]_{n_j}\left \lfloor \alpha\frac{ n_k}{[n_in_j^{-1}]_{n_k}} \right \rfloor \right\}$$
	where $I_k= \left \lceil [-n_in_k^{-1}]_{n_j}\frac{[n_in_j^{-1}]_{n_k}}{n_i} \right \rceil$.
\end{theorem}

\begin{proof}
	By Proposition \ref{pass} $\mathcal{S}_k = \{x-(\lambda_{ik}-1)[-x]_{n_j} | x \in \mathcal{T}_k\}$. However, 
	from Proposition \ref{state} we have that an integer $\alpha n_j + \beta$, with $\alpha \in \mathbb{N}$, $0 < \beta \le n_j$, belongs to $\mathcal{T}_k$ if and only if 
	\begin{equation}\label{min}
	\beta \ge n_j-\frac{(\alpha + 1)n_k}{n_k-\lambda_{ij}}.
	\end{equation}
	Now, write $x \in \mathcal{T}_k$ as $x=\alpha n_j + \beta$, with $0 < \beta \le n_j$. Then $[-x]_{n_j}=[-\beta]_{n_j}=n_j-\beta$, and \begin{equation}\label{val}
	x-(\lambda_{ik}-1)[-x]_{n_j}=\alpha n_j + \beta - (\lambda_{ik}-1)(n_j-\beta)=(\alpha+1-\lambda_{ik})n_j + \lambda_{ik}\beta.
	\end{equation}
	Therefore, for each interval $\mathcal{I}_\alpha = [\alpha n_j + 1 , \alpha n_j + n_j]$, the element $\tilde{x} = \alpha n_j + \tilde{\beta} \in \mathcal{T}_k \cap \mathcal{I}_\alpha$ such that $x-(\lambda_{ik}-1)[-x]_{n_j}$ is minimal is the one with the smallest admissibile remainder, which is, in light of (\ref{min}), $\tilde{\beta}=n_j-\left \lfloor \frac{(\alpha + 1)n_k}{n_k-\lambda_{ij}} \right \rfloor$.

	Putting this value in (\ref{val}) yields $$x-(\lambda_{ik}-1)[-x]_{n_j}=(\alpha+1-\lambda_{ik})n_j + \lambda_{ik}\left( n_j-\left \lfloor \frac{(\alpha+1)n_k}{n_k-\lambda_{ij}} \right \rfloor\right) = (\alpha + 1)n_j- \lambda_{ik}\left \lfloor \frac{(\alpha+1)n_k}{n_k-\lambda_{ij}} \right \rfloor.$$
	Now, since $c_k = \min \mathcal{S}_k$ we deduce $$c_k = \min \mathcal{S}_k \setminus \{0\}=\min_{\alpha \in \mathbb{N}} \{\min \{x-(\lambda_{ik}-1)[-x]_{n_j} \ | \ x \in \mathcal{T}_k \cap I_\alpha) \} \} =$$ $$= \min_{\alpha \in \mathbb{N}} \left \{ (\alpha + 1)n_j- \lambda_{ik}\left \lfloor \frac{(\alpha+1)n_k}{n_k-\lambda_{ij}} \right \rfloor \right\}=\min_{\alpha \in \mathbb{Z}^+} \left \{ \alpha n_j- \lambda_{ik}\left \lfloor \frac{\alpha n_k}{n_k-\lambda_{ij}} \right \rfloor \right\}.$$
	
	Since $n_k - \lambda_{ij}=n_k-[-n_in_j^{-1}]_{n_k}=[n_in_j^{-1}]_{n_k}$, we are left to prove that if $\tilde{\alpha} \in \mathbb{Z}^+$ is such that $c_k=\alpha n_j- \lambda_{ik}\left \lfloor \frac{\alpha n_k}{n_k-\lambda_{ij}} \right \rfloor$, then $\tilde{\alpha} \le \left \lceil \lambda_{ik}\frac{n_k-\lambda_{ij}}{n_i} \right \rceil$. For this purpose, we distinguish two possible cases:
	\begin{enumerate}
		\item $n_j -  \lambda_{ik}\left \lfloor \frac{n_k}{n_k-\lambda_{ij}} \right \rfloor  \ge \lambda_{ik}$.\\
		Recall that the floor function satisfies $\lfloor A+B \rfloor \le \lfloor A \rfloor + \lfloor B \rfloor + 1$ for every $A,B \in \mathbb{R}^+$. By this property, we have that $$\lambda_{ik}\left \lfloor \frac{(\alpha + 1)n_k}{n_k-\lambda_{ij}} \right \rfloor = \lambda_{ik}\left \lfloor \frac{\alpha n_k}{n_k-\lambda_{ij}}+\frac{n_k}{n_k-\lambda_{ij}} \right \rfloor \le \lambda_{ik}\left \lfloor \frac{\alpha n_k}{n_k-\lambda_{ij}} \right \rfloor + \lambda_{ik}\left \lfloor \frac{n_k}{n_k-\lambda_{ij}} \right \rfloor+\lambda_{ik}.$$
		Therefore, for every $\alpha \in \mathbb{Z}^+$ we obtain 
		$$(\alpha+1) n_j -  \lambda_{ik}\left \lfloor \frac{(\alpha+1) n_k}{n_k-\lambda_{ij}} \right \rfloor \ge (\alpha+1)n_j - \lambda_{ik}\left \lfloor \frac{\alpha n_k}{n_k-\lambda_{ij}} \right \rfloor - \lambda_{ik}\left \lfloor \frac{n_k}{n_k-\lambda_{ij}} \right \rfloor-\lambda_{ik} \ge $$ $$\ge \alpha n_j -  \lambda_{ik}\left \lfloor \frac{\alpha n_k}{n_k-\lambda_{ij}} \right \rfloor,$$ and hence by definition of $c_k$ if follows $\tilde{\alpha}=1 \le \left \lceil \lambda_{ik}\frac{n_k-\lambda_{ij}}{n_i} \right \rceil$.
		\item $n_j -  \lambda_{ik}\left \lfloor \frac{n_k}{n_k-\lambda_{ij}} \right \rfloor < \lambda_{ik}.$\\
		 For every $\alpha \in \mathbb{Z}^+$ we have $$\alpha n_j -  \lambda_{ik}\left \lfloor \frac{\alpha n_k}{n_k-\lambda_{ij}} \right \rfloor \ge \alpha n_j -  \lambda_{ik}\frac{\alpha n_k}{n_k-\lambda_{ij}}=\frac{\alpha n_i}{n_k-\lambda_{ij}}.$$
		Assume now that $\tilde{\alpha} > \left \lceil \lambda_{ik}\frac{n_k-\lambda_{ij}}{n_i} \right \rceil \ge \lambda_{ik}\frac{n_k-\lambda_{ij}}{n_i}$. Then 
		$$c_k=\tilde{\alpha} n_j -  \lambda_{ik}\left \lfloor \frac{\tilde{\alpha} n_k}{n_k-\lambda_{ij}} \right \rfloor \ge \frac{\tilde{\alpha} n_i}{n_k-\lambda_{ij}} \ge \lambda_{ik}.$$
		However, considering $\alpha=1$ we obtain $n_j -  \lambda_{ik}\left \lfloor \frac{n_k}{n_k-\lambda_{ij}} \right \rfloor < \lambda_{ik} \le c_k$, which contradicts the definition of $c_k$.
	\end{enumerate}
\end{proof}

Since we do not place any condition (besides the ones stated at the beginning) on the triple $\{n_i,n_j,n_k\}$, we obtain, for each of the integers $c_1,c_2,c_3$, two similar formulas, according to the choice of the indexes $i,j$. Putting these formulas in Theorem \ref{altfrob} yields an (atrociously long) exact formula for $F(\langle n_1,n_2,n_3 \rangle)$, under the general assumption $n_1 > n_2 > n_3$.
Our choice (which we will motivate in the final Section) goes as follow:
\begin{enumerate}
	\item Deduce $c_1$ with the position $i=3,j=2,k=1$, obtaining $$c_1=\min_{\alpha = 1, \ldots, I_1} \left\{\alpha n_2 -  [-n_3n_1^{-1}]_{n_2}\left \lfloor \frac{\alpha n_1}{[n_3n_2^{-1}]_{n_1}} \right \rfloor\right\},$$ where $I_1 = \left \lceil [-n_3n_1^{-1}]_{n_2}\frac{[n_3n_2^{-1}]_{n_1}}{n_3} \right \rceil$.
	\item Compute $c_2$ with the position $i=3,j=1,k=2$. We have $$c_2=\min_{\beta = 1, \ldots, I_2} \left\{\beta n_1 -  [-n_3n_2^{-1}]_{n_1}\left \lfloor \frac{\beta n_2}{[n_3n_1^{-1}]_{n_2}} \right \rfloor\right\},$$ where $I_2 = \left \lceil [-n_3n_2^{-1}]_{n_1}\frac{[n_3n_1^{-1}]_{n_2}}{n_3} \right \rceil$.
	\item Compute $c_3$ with the position $i=2,j=1,k=3$, obtaining $$c_3=\min_{\gamma=1,\ldots,I_3} \left\{\gamma n_1 -  [-n_2n_3^{-1}]_{n_1}\left \lfloor \frac{\gamma n_3}{[n_2n_1^{-1}]_{n_3}} \right \rfloor \right\},$$ where
	$I_3 = \left \lceil [-n_2n_3^{-1}]_{n_1}\frac{[n_2n_1^{-1}]_{n_3}}{n_2} \right \rceil$.
\end{enumerate}
At this point, Theorem \ref{altfrob} returns the following formula.
\begin{theorem}\label{iterfrob}
Let $\mathcal{S}$ be a numerical semigroup minimally generated by three pairwise coprime positive integers $n_1,n_2,n_3$ such that $n_1 > n_2 > n_3$. Then 

$$F(\mathcal{S})=\min_{\alpha = 1, \ldots, I_1} \left\{\alpha n_2 -  [-n_3n_1^{-1}]_{n_2}\left \lfloor \frac{\alpha n_1}{[n_3n_2^{-1}]_{n_1}} \right \rfloor \right\} n_1 + $$ $$+ \max \bigg\{ \left[\min_{\beta = 1,\ldots, I_2} \left\{\beta n_1 -  [-n_3n_2^{-1}]_{n_1}\left \lfloor \frac{\beta n_2}{[n_3n_1^{-1} ]_{n_2} } \right \rfloor  \right\}n_2n_3^{-1} \right]_{n_1}n_3, $$ $$ \left[\min_{\gamma = 1 ,\ldots, I_3} \left\{\gamma n_1 -  [-n_2n_3^{-1}]_{n_1}\left \lfloor \frac{\gamma n_3}{[n_2n_1^{-1}]_{n_3}} \right \rfloor \bigg\} n_3n_2^{-1} \right]_{n_1}n_2 \right\} -n_1-n_2-n_3,$$ where $I_1 = \left \lceil [-n_3n_1^{-1}]_{n_2}\frac{[n_3n_2^{-1}]_{n_1}}{n_3} \right \rceil$, $I_2 = \left \lceil [-n_3n_2^{-1}]_{n_1}\frac{[n_3n_1^{-1}]_{n_2}}{n_3} \right \rceil$ and $I_3 = \left \lceil [-n_2n_3^{-1}]_{n_1}\frac{[n_2n_1^{-1}]_{n_3}}{n_2} \right \rceil$.
\end{theorem}

\section{Conclusions}
Clearly, the formula of Theorem \ref{iterfrob} is mainly a theoretical exercise, whose utility is to directly highlight the complexity of the Diophantine Frobenius problem and the extremely irregular behaviour of $F(\langle n_1,n_2,n_3 \rangle)$; it does not give generic bounds, or asymptotic results. However, it may prove useful as a computable formula that gives $F(\mathcal{S})$ only in function of the generators: in fact, since the modular parts of the form $\lambda_{ij}=[-n_in_j^{-1}]_{n_k}$ repeat themselves in the formula, it may sound preferable to compute them once at the beginning, and use them directly in this formula, rather than compute them each time we recall a function based on Theorem \ref{main}. In particular, notice that the choice that led to Theorem \ref{iterfrob} allows us to compute only four of these $\lambda_{ij}$, rather than six.

On the other hand, Theorem \ref{main} may have practical uses (although it is still susceptible to improvements) and provides some insight on the problem. In general, the integer $c_k$ is defined as the minimum positive integer satisfying a certain modular inequality (cf. Proposition \ref{eq}); in Theorem \ref{main}, this condition disappeared; another advantage is that the set containing $c_k$ is finite and computable. Moreover, by looking at the formula, the most complex part seems to be the computation of the modular moltiplicative inverse of $n_k$ modulo $n_j$, which has the same complexity as the extended Euclidean algorithm. However, since these modular inverses do not involve $\alpha$, they have to be computed only once.  

Moreover, the complexity of the computation of $c_1,c_2,c_3$ can be reduced by using some tricks. For instance, it is possible to rotate the indexes $\{i,j,k\}$ in a way that allows to use the same inverses for two different integers $c_k$ (see the choice that led to Theorem \ref{iterfrob}). Another possible choice is to fix $k$ and then take the choice of $i,j$ such that the number of iterations required is minimal. 

Probably, it is possible to compute the value of $\alpha$ that gives $c_k$, thus improving Theorem \ref{iterfrob}: however, it seems that investigating this $\alpha$ with elementary methods would only lead to an even more convoluted result. At least, continuing in that direction will require a new idea.

In some particular settings, Theorem \ref{main} can be presented in a simplified form. One of these has already been considered in the proof of Theorem \ref{main}.
\begin{corollary}\label{cor}
		With the notation used in Theorem \ref{main}, if $ n_j \ge [-n_in_k^{-1}]_{n_j}\left( \left \lfloor \frac{ n_k}{[n_in_j^{-1}]_{n_k}} \right \rfloor + 1 \right)$, we have $$c_k= n_j - [-n_in_k^{-1}]_{n_j}\left \lfloor \frac{ n_k}{[n_in_j^{-1}]_{n_k}} \right \rfloor.$$
\end{corollary}

This last corollary can simplify the computation if $\lambda_{ik},\lambda_{ij}$ are small enough: for instance, if $\lambda_{ij} < \frac{n_k}{2}$ and $\lambda_{ik} < \frac{n_j}{2}$, $n_j -  \lambda_{ik}\left \lfloor \frac{n_k}{n_k-\lambda_{ij}} \right \rfloor = n_j - \lambda_{ik} \ge \lambda_{ik}$, thus we can compute the two minimal relations $c_k$ and (swapping $j$ and $k$) $c_j$ by Corollary \ref{cor}, speeding up the process. 

The hypotheses of Corollary \ref{cor} are not always satisfied (see for example $n_1=12,n_2=11,n_3=7$); but, since $0 < n_i=n_jn_k-\lambda_{ij}n_j-\lambda_{ik}n_k$ we must have either $\lambda_{ij} < \frac{n_k}{2}$ or $\lambda_{ik} < \frac{n_j}{2}$. Thus, since $0 < \lambda_{ij} < n_k$ and $0 < \lambda_{ik} < n_j$, it is easy to understand that Corollary \ref{cor} can be applied for a large number of triples.

\end{document}